\documentclass[12pt]{article}
\input xy
\xyoption{all}

\usepackage[colorlinks=true,linkcolor=blue,citecolor=blue]{hyperref}

\usepackage{marvosym}
\usepackage{amssymb}
\usepackage{amsthm}
\usepackage{enumerate}
\usepackage[active]{srcltx}
\usepackage{amsmath}
\usepackage{nicefrac}
\usepackage{latexsym}
\usepackage{graphicx}
\usepackage{textcomp}
\numberwithin{equation}{section}
\theoremstyle{plain}
  \begingroup
        \newtheorem{theorem}[equation]{Theorem}
        \newtheorem{lemma}[equation]{Lemma}
        \newtheorem{proposition}[equation]{Proposition}

	\newtheorem{defi}[equation]{Definition}

  \endgroup

\newcommand{\comments}[1]{}

\theoremstyle{definition}
  \begingroup
        \newtheorem{remark}[equation]{Remark}

  \endgroup

\newcommand{\bb}{\mathbb}

\def\xto#1{\xrightarrow{#1}}

\newcommand{\sbt}{\,\begin{picture}(-1,1)(-1,-3)\circle*{3}\end{picture}\ }

\title{\large THE QUASI-STATE SPACE OF A $C^*$-ALGEBRA IS A TOPOLOGICAL QUOTIENT OF THE REPRESENTATION SPACE}

\author{\normalsize Sergio A. Yuhjtman}

\begin{document}
\maketitle

\begin{abstract}
  We show that for any $C^*$-algebra $A$, a sufficiently large Hilbert space $H$ and a unit vector $\xi \in H$,
  the natural application \mbox{$rep(A:H) \xto {\theta_\xi} Q(A)$}, \mbox{$\pi \mapsto \langle \pi(-)\xi,\xi \rangle$} is a topological quotient,
  where $rep(A:H)$ is the space of representations on $H$ and $Q(A)$ the set of quasi-states, 
  i.e. positive linear functionals with norm at most $1$. This quotient might be a useful tool in the representation 
  theory of \mbox{$C^*$-algebras}. We apply it to give an interesting proof of Takesaki-Bichteler 
  duality for $C^*$-algebras which allows to drop a hypothesis.
\end{abstract}

\section{Introduction}

  $ \ \ $ The GNS construction relates positive linear functionals with cyclic representations of a $C^*$-algebra. If we take a Hilbert
  space $H$ and a unit vector $\xi \in H$, it is natural to consider 
  the map $rep(A:H) \xto {\theta_\xi} Q(A)$, $\pi \mapsto \langle \pi(-)\xi,\xi \rangle$. If
  $H$ is large enough to contain (strictly) a copy of every cyclic representation, the GNS construction 
  is essentially equivalent to the surjectivity of $\theta_\xi$. 
  Considering the $weak^*$ topology on $Q(A)$ and the correct topology in $rep(A:H)$ described below, the map $\theta_\xi$ is continuous.
  Here we show that this map is a topological quotient (theorem \ref{main}a). This property provides a more complete picture of the relationship
  between these fundamental objects of the theory. As an application, we present an interesting perspective for 
  Takesaki-Bichteler duality, that is summarized
  by the diagram in the proof of theorem \ref{takduality}. 
  Despite our proof of the duality preserves two key ingredients from Bichteler's proof
  (proposition 4.ii and first lemma in \cite{Bich}), it is conceptually more clear and it allows to drop the boundedness
  condition (definition \ref{field}, 2). From the proofs in \cite{Bich} and \cite{Tak} it is not possible to directly avoid such hypothesis.
  We also review the concept of ``field'' relevant in this context, giving a more elegant 
  definition and explaining the equivalence with the old ones. 
  Thus, conditions in theorem \ref{takduality} are significantly 
  better than those imposed to the fields in \cite{Tak} and \cite{Bich}.
  
  For unital $C^*$-algebras we also show that the state space $S(A)$ is a topological quotient of the appropriate subspace
  of $rep(A:H)$. This is \ref{main}b.

  Our application of theorem \ref{main} to Takesaki-Bichteler duality, only exploits the universal 
  property of the quotient in the case of affine scalar maps. We expect the existence of other applications
  where the involved maps are nonlinear.

\subsection{Notation}

\noindent \sbt $ $ $A$ will denote a $C^*$-algebra.

\noindent \sbt $ $ If $X$ is a Banach space, $X^*$ denote its dual.

\noindent \sbt $ $ $S(A)=\{\varphi \in A^* / \varphi \geq 0, \ \|\varphi\|=1 \}$ the state space of $A$, with the weak-* topology.

\noindent \sbt $ $ $Q(A)=\{\varphi \in A^* / \varphi \geq 0, \ \|\varphi\| \leq 1\}$, the space of quasi-states, also with the weak-* topology.

\noindent \sbt $ $ For $\varphi \in A^*$, $\varphi \geq 0$, $(\pi_\varphi, H_\varphi, \xi_\varphi)$ is the GNS triple. $\|\xi_\varphi\|^2=\|\varphi\|$.

\section{Main theorem}

\hspace{.5cm} Let $rep(A:H)$ be the set of possibly degenerate representations of $A$ on $H$, this is the set of $*$-algebra morphisms $A \to B(H)$.
Here $H$ is a Hilbert space of a dimension greater or equal than the dimension of every cyclic representation of $A$.
In case the supremum of these dimensions is finite, we will require the dimension of $H$ to be strictly larger than this number.

The relevant topology on $rep(A:H)$ is the pointwise convergence topology with respect to the $wot$, $sot$, $\sigma$-weak or $\sigma$-strong
topologies in $B(H)$. Next lemma asserts that these topologies coincide.

\begin{lemma}
 Let $\pi$ be a representation of $A$ on a Hilbert space $H$ and $(\pi_j)$ a net of such representations. Convergence
$\pi_j(a) \to \pi(a)$ for all $a \in A$ is equivalent for the $wot$, $sot$, $\sigma$-weak and $\sigma$-strong topologies on $B(H)$.
\end{lemma}

See \cite{Bich} for the proof (page 90). 
In other words, the topology we consider on $rep(A:H)$ is that inherited from the product topology on $B(H)^A$, where the topology on
$B(H)$ can equivalently be the $\sigma$-weak, $\sigma$-strong, $wot$ or $sot$. It is Hausdorff because it is a subspace of a product
of Hausdorff spaces.

For the proof of theorem \ref{main} we will need the following simple geometrical lemma.

\begin{lemma} \label{lemita}
 Let $H$ be a Hilbert space and $\alpha, \beta \in H$ unit vectors. Then there is a unitary
$U_{\alpha \to \beta}$ such that $U_{\alpha \to \beta}(\alpha)=\beta$ and
$\|U_{\alpha \to \beta}-Id\|=\|\alpha-\beta\|$.
\end{lemma}

\begin{proof}
 In case $\beta = k \alpha$ for $k \in \bb C$, then $|k|=1$ and $U_{\alpha \to \beta}:=k.Id$. 
Otherwise, we define $U_{\alpha \to \beta}$ as the identity on $[\alpha]^\perp \cap [\beta]^\perp=[\alpha,\beta]^\perp$.
On the subspace $[\alpha,\beta]$
we take an orthonormal basis $(\alpha,\alpha')$. Write $\beta=r\alpha + s \alpha'$, and \mbox{$\beta':=-\overline{s}\alpha +\overline{r} \alpha'$}, obtaining
an orthonormal basis $(\beta,\beta')$.
Now define $U_{\alpha \to \beta}|_{[\alpha,\beta]}$ by $\alpha \mapsto \beta$, $\alpha' \mapsto \beta'$. The following two identities
are easy to check:
$$\langle \alpha - \beta, \alpha' - \beta' \rangle=0$$
$$\| \alpha - \beta \|= \| \alpha' - \beta'\|$$
For $x \in H$, let $\lambda \alpha + \mu \alpha'$ be the projection of $x$ to $[\alpha,\beta]$. We have:

$$\|x-U_{\alpha \to \beta}(x)\|^2=\|\lambda \alpha + \mu \alpha' - \lambda \beta - \mu \beta'\|^2=$$
$$=\| \lambda(\alpha-\beta)+\mu(\alpha' - \beta')\|^2=(|\lambda|^2+|\mu|^2)\|\alpha-\beta\|^2$$

So $\|x-U_{\alpha \to \beta}(x)\|=\|\alpha-\beta\|.\|p_{[\alpha,\beta]}(x)\| \leq \|\alpha-\beta\|.\|x\|$.
\end{proof}

The proof of our main theorem makes use of the following proposition by Bichteler. This is proposition 4) (ii) in \cite{Bich}.

\begin{proposition} \label{key}
 Let $A$ be a unital $C^*$-algebra and $H$ a Hilbert space large enough to contain a copy of any cyclic representation of $A$.
 Let $\varphi \in Q(A)$, $\pi \in rep(A:H)$, $\xi \in \pi(1)H$ such that $\langle \pi(-)\xi,\xi \rangle=\varphi$. 
 
 \noindent For every \mbox{$V \subset rep(A:H)$} and $W \subset H$
 open neighborhoods of $\pi$ and $\xi$ respectively, there is an open neighborhood $U$ of $\varphi$ such that for every $\varphi' \in U$
 there are $\pi' \in V$, $\xi' \in W \cap \pi'(1)H$ satisfying \mbox{$\langle \pi'(-)\xi',\xi' \rangle=\varphi'$}.
\end{proposition}

In the following proposition we manage to keep fixed the vector $\xi'$ in previous statement. We require $H$ to contain 
strictly a copy of any cyclic representation of $A$. Of course, this detail only makes a difference when $H$ is finite dimensional.
Besides, we need $\pi(1)H \subsetneq H$.

\begin{proposition} \label{saraza}
 Let $A$ be a unital $C^*$-algebra and $H$ a Hilbert space of dimension $d$, greater or equal than the dimension of
any cyclic representation of $A$, plus $1$.
Let $\xi \in H$ be a unit vector. Then, for every $\pi \in rep(A:H)$ such that $\pi(1)H \subsetneq H$
and $V \ni \pi$ open, $\theta_\xi(V)$ is a neighborhood of $\theta_\xi(\pi)=\varphi$.
\end{proposition}

\begin{proof}
 We might assume $V=\{\pi' \in rep(A:H) / \|\pi'(a_i)\xi_j - \pi(a_i)\xi_j\| < \epsilon\}$ for finite families $(a_i)$, $(\xi_j)$ with $\|a_i\|=\|\xi_j\|=1$.
 Let $H_\pi:=\pi(1)H=\overline{\pi(A)H}$. Let $\eta=\pi(1)\xi$. We have $\xi - \eta \perp H_\pi$. Take $H' \subsetneq H$ such 
 that $H_\pi \subset H' \subset [\xi-\eta]^\perp$ and $H'$ contains a copy of every cyclic representation of $A$. Now take 
 $$V':=\{\pi' \in rep(A:H') / \|\pi'(a_i)\xi_j - \pi(a_i)\xi_j\| < \frac{\epsilon}{2}\}$$ 
 (we assume $\pi'(a)(H'^\perp)=0$ for $\pi' \in rep(A:H')$, so $V'$ is an open subset of $rep(A:H')$ containing $\pi$).
 
 According to proposition \ref{key}, if we take $W \subset H'$ the $\delta$-ball centered at $\eta$, 
 there is an open set $U \ni \varphi$ such that for any $\varphi' \in U$ there is a $\pi' \in V'$ and
 $\eta' \in W \cap \pi'(1)H'$ satisfying $\langle \pi'(-)\eta',\eta' \rangle =\varphi'$. (Note: we can choose $U$ such that $|\varphi'(1)-\varphi(1)|<\epsilon_1$ $\forall \varphi' \in U$,
 so $\Big\||\varphi'\|-\|\varphi\|\Big|<\epsilon_1$).
 Now we only need to rotate $\pi'$ slightly by a unitary $U$ in such a way that $\langle U^{-1}\pi'(-)U\xi,\xi \rangle =\varphi'$.
 The image of $\xi$ by $U$ must be a unit vector $\xi'$ close to $\xi$ such that \mbox{$\xi'-\eta' \perp \pi'(1)H$}.

 In case $\xi=\eta$, since $\|\eta'\| \leq 1$,
we can take $v \in H'^\perp$ such that $\|\eta'+v\|=1$, and define $\xi':=\eta'+v$. We have
$$\|\xi-\xi'\|^2=\|\eta-(\eta'+v)\|^2=\|\eta-\eta'\|^2+\|v\|^2= \|\eta-\eta'\|^2 + 1 - \|\eta'\|^2$$

Since $\|\eta-\eta'\| < \delta$ we have $\|\eta'\|>1-\delta$, and we can easily make (taking $\delta$ sufficiently small) $\|\xi-\xi'\|<\epsilon_2$, to be determined.

In case $\xi \neq \eta$, we take $\xi'= \eta'+\lambda(\xi-\eta)$. To determine $\lambda$:
$$\|\xi'\|^2=|\lambda|^2.\|\xi-\eta\|^2+\|\eta'\|^2=|\lambda|^2.(1-\|\eta\|^2)+\|\varphi'\|=|\lambda|^2.(1-\|\varphi\|)+\|\varphi'\|$$
so we choose $\lambda = (\frac{1-\|\varphi'\|}{1-\|\varphi\|})^{\frac{1}{2}}$ to obtain $\|\xi'\|=1$. Since $\|\varphi'\|$ is arbitrarily close to $\|\varphi\|$,
 $\lambda$ is arbitrarily close to $1$ and therefore $\xi'$ is arbitrarily close to $\xi$ ($\|\xi-\xi'\|<\epsilon_2$) as long as $\delta$ is
sufficiently small.

  Now, having $\xi'$ we just apply $U:=U_{\xi \to \xi'} \in U(H)$ (lemma \ref{lemita}), and take $\pi''(-):=U^{-1}\pi'(-)U$. 
  Since $\pi' \in V$ and $\|U-Id\| = \|\xi' - \xi\|<\epsilon_2$, we have:
  
  $$\|\pi''(a_i)\xi_j - \pi(a_i)\xi_j\|<\|\pi''(a_i)\xi_j - \pi'(a_i)\xi_j\|+ \frac{\epsilon}{2}$$
  
 \noindent and  
  $$\|\pi''(a_i)\xi_j - \pi'(a_i)\xi_j\|=\|U^{-1}\pi'(a_i)U\xi_j-\pi'(a_i)\xi_j\| \leq$$
  $$\|U^{-1}\pi'(a_i)U\xi_j - U^{-1}\pi'(a_i)\xi_j\| + \|U^{-1}\pi'(a_i)\xi_j - \pi'(a_i)\xi_j\| < 2\epsilon_2 < \frac{\epsilon}{2}$$
for $\epsilon_2 < \frac{\epsilon}{4}$. So we get $\pi'' \in V$ and $\theta_\xi(\pi'')=\varphi'$.
\end{proof}

\begin{theorem} \label{main}
 Let $A$ be a $C^*$-algebra and $H$ a Hilbert space of dimension $d$, large enough to contain strictly a copy of
any cyclic representation of $A$.
Let $\xi \in H$ be a unit vector. Then, 

(a) the application
$$rep(A:H) \stackrel{\theta_\xi}{\longrightarrow} Q(A)$$
$$\hspace{70pt} \pi \longmapsto \langle \pi(-)\xi,\xi \rangle$$
is a quotient map.

(b) for unital $A$, the restriction $rep_{\xi}(A:H) \xto{\theta_\xi} S(A)$ is a quotient, where

\noindent $rep_{\xi}(A:H)=\{\pi \in rep(A:H) / \xi \in \pi(1)H\}$.

\end{theorem}

\begin{proof} $ $
 Continuity of $\theta_\xi$ is trivial. Despite surjectivity may be intuitively clear from the GNS construction,
we will describe in detail a generic preimage for $\varphi \in Q(A)$.
We must embed a GNS representation of $\varphi$
in $H$ in such a way that the orthogonal projection of $\xi$ to the essential space is the cyclic vector of the GNS triple. 
To achieve this, take a unit vector 
$v$ orthogonal to $\xi$, define $\eta=\|\varphi\|\xi + (\|\varphi\|-\|\varphi\|^2)^{\frac{1}{2}}v$. 
This $\eta$ satisfies $\|\eta\|^2=\|\varphi\|$ and $\xi-\eta \perp \eta$. By hypothesis, it is possible to embed
$H_{\varphi}$ into $[\xi-\eta]^\perp$ taking $\xi_{\varphi}$ to $\eta$.
Define $\pi \in rep(A:H)$ as $\pi_{\varphi}$ through the
isometry $H_{\varphi} \hookrightarrow H$, being $0$ on the orthogonal to the image of $H_{\varphi}$. We have $\theta_\xi(\pi)=\varphi$.

Now we assume $A$ unital and postpone the general case, because we need part (b). Take $D \subset Q(A)$ such that $V:=\theta_\xi^{-1}(D)$ is open. We must see that $D$ is open to conlude 
that $\theta_\xi$ is a quotient map. Let $\varphi \in D$. 
Take a preimage $\pi$ of $\varphi$ such that $\pi(1)H \subsetneq H$. By proposition \ref{saraza}, $\theta_\xi(V)$ is a neighborhood
of $\varphi$, so $D$ is open.

(b) Clearly we have the restriction $rep_\xi(A:H) \xto {\theta_\xi} S(A)$. Furthermore $\theta_\xi^{-1}(S(A))=rep_\xi(A:H)$. 
Let $D \subset S(A)$ be a set such that $\theta_\xi^{-1}(D)$ is open in $rep_\xi(A:H)$, so $\theta_\xi^{-1}(D)=V \cap rep_\xi(A:H)$
with $V$ open in $rep(A:H)$. Let $\varphi \in D$. We take $\pi \in \theta_\xi^{-1}(\varphi)$ such that $\pi(1)H \subsetneq H$, as before. 
By proposition \ref{saraza}, $\theta_\xi(V)$ contains an open neighborhood $U \ni \varphi$, $U$ open in $Q(A)$. Now it is easy to check:
$$U \cap S(A) \subset \theta_\xi(V\cap rep_\xi(A:H)) \subset D$$

Thus, $D$ is open in $S(A)$.

Finally, we prove the general case of (a). Consider the minimal unitization $\widetilde A$. By restriction, we have a continuous
map $rep_\xi(\widetilde A:H) \xto r rep(A:H)$. Besides, the restriction $S(\widetilde A) \to Q(A)$ is a homeomorphism. We have:

$$\xymatrix{rep_\xi(\widetilde A:H) \ar[r] \ar[d]_r & S(\widetilde A) \ar[d]^{\simeq} \\
rep(A:H) \ar[r]^(.6){\theta_\xi} & Q(A)}$$
Since $\theta_\xi \circ r$ is a quotient by part (b), $\theta_\xi$ is a quotient.
\end{proof}

\begin{remark}
 In case $d$ is finite, $rep(A:H)$ is compact, so $\theta_\xi$ is a closed map. To see compactness of 
$rep(A:H)$, consider the map \mbox{$rep(A:H) \to B_1^{A_1}$}, $\pi \mapsto (a \mapsto \pi(a))$, where $B_1 \subset B(H)$ and $A_1 \subset A$ are
the respective unit balls. $B_1^{A_1}$ has the product topology of the norm topology in $B_1$, it is a compact space. 
The map is a topological subspace and the image is closed. 
\end{remark}

\section{Application to Takesaki-Bichteler duality}

\hspace{.5cm} Takesaki-Bichteler duality allows to recover an arbitrary \mbox{$C^*$-algebra} from its representation theory.
The elements of $A$ coincide with the set of certain continuous fields on $rep(A:H)$. Here we remove the
boundedness hypothesis on the fields, clarify the remaining conditions and present an elegant proof through theorem
\ref{main}.

See \cite{Fujimoto} for very interesting duality theorems with $rep(A:H)$ replaced by $Irr(A:H)$, the space 
of irreducible representations. These dualities are not only related to Gelfand duality but also to Tannaka duality
for compact groups or its generalization to locally compact groups, Tatsuuma duality.

Let us start by reviewing the concept of field that is used in Takesaki-Bichteler's theorem. 
We provide a simpler definition than those from \cite{Tak} and \cite{Bich}, and explain why it is equivalent.

\begin{defi} \label{field}
A field over $rep(A:H)$ is a map $rep(A:H) \xto T B(H)$ that satisfies:

\noindent 0) $T(0)=0$

\noindent 1) For an intertwiner $H \xto S H$ between $\pi_1$ and $\pi_2$ ($S \pi_1(a) = \pi_2(a)S$ $\forall a \in A$), it holds 
$S T(\pi_1)=T(\pi_2)S$. In other words: $T$ is compatible with intertwiners.

\noindent 2) $\{\|T (\pi)\|\}_{\pi \in rep(A:H)}$ is bounded.

\end{defi}

Clearly, every $a \in A$ induces a field.

\begin{proposition} \label{isosparciales}
 The following condition is equivalent to item 1 in previous definition.
 
\noindent 1') $T$ is compatible with intertwiners that are partial isometries. 

\end{proposition}

\begin{proof}
  $1) \Rightarrow 1')$ is trivial, let us prove the converse. 
 Assume that $T$ is compatible with intertwiners that are partial isometries and take an arbitrary intertwiner $\pi_1 \xto S \pi_2$, $S \in B(H)$.
 The operator $S$ has a polar decomposition $S=UP$, where $P=(S^*S)^{\frac{1}{2}}$ and $U$ maps $(S^*S)^{\frac{1}{2}}y$
 to $Sy$ and the orthogonal complement to $0$. Since $S$ is an intertwiner, $\pi_2 \xto{S^*} \pi_1$ is an intertwiner and also are
 $\pi_1 \xto P \pi_1$ and $\pi_1 \xto U \pi_2$. $T$ is compatible with $U$ by hypothesis. 
 It only remains to prove that $T$ is compatible with any positive 
 intertwiner $P$ of a representation $\pi_1$ with itself. Taking $r>0$ small enough, $rP$ has its spectrum inside $[0,2\pi)$. $e^{irP}$ is a
 unitary equivalence, so it is compatible with $T$ (i.e. it commutes with $T(\pi_1)$). But $rP$ is the logarithm of $e^{irP}$, so $rP$
 also commutes with $T(\pi_1)$.

\end{proof}

\begin{proposition} \label{essential}
 Let $\pi \in rep(A:H)$, $p_\pi$ the orthogonal projection to the essential space of $\pi$ and $T$ a field over $rep(A:H)$. Then
 $T(\pi)=p_\pi T(\pi) p_\pi$.
\end{proposition}

\begin{proof}
  Let $p_{\pi^{\perp}} = 1 - p_\pi$, the orthogonal projection to $\overline{\pi(A)H}^\perp$. It defines an intertwiner
  $\pi \xto{p_{\pi^{\perp}}} 0$, so we have $(1-p_\pi)T(\pi)=T(0)p_{\pi^\perp}=0$, $T(\pi)=p_\pi T(\pi)$. Besides, $p_\pi$ is an endomorphism of $\pi$,
  so $p_\pi T(\pi)=T(\pi) p_\pi$.
\end{proof}

With the operations defined pointwise and the norm $\|T\|=\sup_\pi \|T(\pi)\|$, the set of fields is a $C^*$-algebra. Actually,
they form the universal von Neumann algebra of $A$ (see \cite{Tak} theorem 3, 
\cite{Bich} proposition in page 95, \cite{tesis} proposition 4.7). Recall that the universal von Neumann algebra of a \mbox{$C^*$-algebra}
$A$ can also be constructed as the bicommutant of the universal representation $\bigoplus_{\varphi \in S(A)} \pi_\varphi$
or as the bidual $A^{**}$ with the natural involution and Arens multiplication.

The definition of ``field'' by Takesaki can be summarized as follows: it is a bounded map $rep(A:H) \xto T B(H)$ with the property 
$T(\pi)= p_\pi T(\pi) p_\pi$, compatible with unitary equivalences (in the sense of our condition \ref{field}.1)
and finite direct sums. For direct sums, it is necessary to consider a unitary $H \oplus H \xto J H$, so the condition can be expressed:
$$Ad \hspace{.1cm} J \hspace{.1cm} (T(\pi_1) \oplus T(\pi_2))= T(Ad \hspace{.1cm} J \hspace{.1cm} (\pi_1 \oplus \pi_2))$$
$$\mbox{where } Ad \hspace{.1cm} J \hspace{.1cm} (-) := J(-)J^*$$
Our defintion is stronger because we have compatibility with arbitrary intertwiners and \ref{essential}. 
The converse can be done through proposition \ref{isosparciales}:
a field compatible with direct sums and unitary equivalences will be compatible with intertwiners that are partial isometries.
We prefer not to write down the details. Actually, it is technically unnecessary, since we already know that both definitions 
lead to the enveloping von Neumann algebra of $A$.

\bigskip

 Takesaki-Bichteler duality asserts that a $C^*$-algebra can be recovered as the set of continuous fields $rep(A:H) \to B(H)$, where the 
topology on $B(H)$ might be the $\sigma$-weak, $\sigma$-strong, $wot$ or $sot$. Elements in $A$ clearly induce continuous fields for 
all these topologies on $B(H)$. Since $wot$ is the weakest among these, we have that $sot$-continuous, \mbox{$\sigma$-weak-continuous}
and \mbox{$\sigma$-strong-continuous} fields are $wot$-continuous. 
Hence, it will suffice to prove that \mbox{$wot$-continuous} fields are elements of $A$.

In order to deduce the duality theorem from theorem \ref{main}, we need the following lemma taken from Bichteler's article 
(\cite{Bich}, first lemma, parts (iii) and (iv)).

Recall that any Banach space $V$ can be recovered from the bidual as those elements $V^* \to \bb C$ that are continuous 
for the $w^*$-topology. This lemma in particular says that for a $C^*$-algebra $A$ it suffices with continuity on $Q(A)$ 
instead of all $A^*$.

\begin{defi}
 Let $AN_0(Q(A))$ be the set of affine bounded $\bb C$-valued functions on $Q(A)$ taking the value $0$ at $0$.
It is a normed space for the supremum norm. $AC_0(Q(A))$ will be the subspace of $AN_0(Q(A))$ of continuous functions.
\end{defi}

\begin{lemma} \label{lema}
 There is a Banach space isomorphism $A^{**} \to AN_0(Q(A))$ that restricts to a bijection $A \to AC_0(Q(A))$.
\end{lemma}

\begin{proof}
 The map $A^{**} \to AN_0(Q(A))$ is defined by restriction from $A^*$ to $Q(A)$. It is straightforward to check that it is a Banach space
 isomorphism (see \cite{Bich} first lemma or \cite{tesis} lemma 5.2). For the second part, we prefer the following proof instead of the one from
 \cite{Bich}. 
 
  Through the isomorphism we have $A \subset AC_0(Q(A))$. \mbox{We must prove} that equality holds. Take $f \in AC_0(Q(A))$. 
  We \mbox{have continuous maps}
\mbox{$Q(A) \times Q(A) \xto{\bar{f}} \bb C$}, \mbox{$\bar{f}(\varphi,\psi)=f(\varphi)-f(\psi)$}, and \mbox{$Q(A) \times Q(A) \xto m A^*_h$},
\mbox{$m(\varphi,\psi)=\varphi-\psi$} ($A^*_h$ is the hermitian part of $A^*$ with the $w^*$-topology). Since $Q(A) \times Q(A)$ is 
compact and $A^*_h$ Hausdorff, $m$ is closed, and therefore a quotient if we restrict the codomain to the image.

$$\xymatrix{Q(A) \times Q(A) \ar[r]^(.7)m \ar[dr]_{\bar{f}} & A^*_h \ar[d]^{\tilde f} \\ & \bb C}$$

The image of $m$ contains the unit ball, because every $\varphi \in A^*_h$ can be written as $\varphi_1 - \varphi_2$ with
$\varphi_1,\varphi_2 \geq 0$ and $\|\varphi\|=\|\varphi_1\|+\|\varphi_2\|$. Thus, $\tilde f$ is $w^*$-continuous on the unit ball. As a consequence of
Krein-Smulian theorem, $\tilde f$ is continuous on $A^*_h$. Analogously, it is continuous on $A^*_{ah}$ and therefore on $A^*$. 
Hence we conclude that $f$ comes from an element of $A$. 
\end{proof}

\begin{remark} \label{S(A) en vez de Q(A)}
 Taking $S(A)$ instead of $Q(A)$ we have: $A^{**} \simeq AN(S(A))$ and, for unital $A$, $A \simeq AC(S(A))$ (where $AN(S(A))$ is the
space of affine bounded $\bb C$-valued functions on $S(A)$ and $AC(S(A))$ the subspace of continuous functions).
\end{remark}

\begin{proof}
 It is straightforward to check that $AN_0(Q(A)) = AN(S(A))$. To obtain $AC_0(Q(A)) = AC(S(A))$ for unital $A$, we must prove that continuity on $S(A)$ implies continuity on $Q(A)$.
So take $f \in AN_0(Q(A))$ continuous on $S(A)$ and $\varphi_\mu \to \varphi$ in $Q(A)$. Evaluating at $1$, we have $\|\varphi_\mu\| \to \|\varphi\|$.
If $\varphi=0$ we have $|f(\varphi_\mu)|=\|\varphi_\mu\|.|f(\frac{\varphi_\mu}{\|\varphi_\mu\|})| \leq \|\varphi_\mu\|.\|f\|_{\infty} \to 0$ for those 
$\mu$ such that $\varphi_\mu \neq 0$ and $f(\varphi_\mu)=0$ if $\varphi_\mu = 0$; so $f(\varphi_\mu) \to 0$. If $\varphi \neq 0$, for large
enough $\mu$ we have $\varphi_\mu \neq 0$ and

$$f(\varphi_\mu)=\|\varphi_\mu\|f(\frac{\varphi_\mu}{\|\varphi_\mu\|}) \to \|\varphi\|f(\frac{\varphi}{\|\varphi\|})=f(\varphi)$$
\end{proof}

\begin{theorem}[Takesaki-Bichteler duality] \label{takduality}
 Every $C^*$-algebra $A$ is isomorphic to the set of $wot$-continuous maps $rep(A:H) \xto T B(H)$ 
 compatible with intertwiners that are partial isometries and such that $T(0)=0$. 
 Here $H$ is a Hilbert space large enough to contain strictly a copy of any cyclic representation of $A$.
\end{theorem}

\begin{proof}
 We already know that an element of $A$ defines a continuous field. Now take a $wot$-continuous $T$ as in the
 statement.

Since $\theta_\xi$ is surjective, there is an only way to define $f_T$ in order to make the following square commutative:

$$\xymatrix{rep(A:H) \ar[r]^(.6){\theta_\xi} \ar[d]_T & Q(A) \ar@{.>}[d]^{f_T} \\
B(H) \ar[r]^{\langle(-)\xi,\xi\rangle} & \bb C}$$

Compatibility of $T$ with intertwiners allows to prove that $f_T$ is well defined
and affine in a straightforward manner. Since $T(0)=0$, we have $f_T(0)=0$. Continuity of $T$ 
implies continuity of $f_T$ because $\theta_\xi$ is a topological quotient (theorem \ref{main}a). By lemma \ref{lema}, $f_T$ is an element of $A$.

\end{proof}

\begin{remark}
 For unital $A$, a field over $rep(A:H)$ only needs to be continuous on $rep_\xi(A:H)$ to be an element of $A$. This is because
of part (b) of theorem \ref{main} and remark \ref{S(A) en vez de Q(A)}.
\end{remark}

{\bf Acknowledgements:} I want to thank Rom\'an J. Sasyk for the support along these years.

\bibliographystyle{amsplain}
\bibliography{bib}

\def\cprime{$'$}
\providecommand{\bysame}{\leavevmode\hbox to3em{\hrulefill}\thinspace}
\providecommand{\MR}{\relax\ifhmode\unskip\space\fi MR }
\providecommand{\MRhref}[2]{%
  \href{http://www.ams.org/mathscinet-getitem?mr=#1}{#2}
}
\providecommand{\href}[2]{#2}
\begin{thebibliography}{1}

\bibitem{Bich}
Klaus Bichteler, \emph{{A generalization to the non-separable case of
  {T}akesaki's duality theorem for {$C^{\ast} $}-algebras}}, Invent. Math.
  \textbf{9} (1969/1970), 89--98. \MR{0253060 (40 \#6275)}

\bibitem{Fujimoto}
Ichiro Fujimoto, \emph{{A {G}elfand-{N}aimark theorem for {$C^*$}-algebras}},
  Pacific J. Math. \textbf{184} (1998), no.~1, 95--119. \MR{1626504
  (99b:46083)}

\bibitem{Tak}
Masamichi Takesaki, \emph{{A duality in the representation theory of {$C^{\ast}
  $}-algebras}}, Ann. of Math. (2) \textbf{85} (1967), 370--382. \MR{0209859
  (35 \#755)}

\bibitem{tesis}
Sergio~A. Yuhjtman, \emph{{Duality theorems for C*-algebras, multiplier algebra
  in duality contexts, Tietze extension theorem and related results}}, Ph.D.
  thesis, Universidad de Buenos Aires, 2013.

\end{thebibliography}

\bigskip

\bigskip

\bigskip

\small

{\bf Sergio Andr\'es Yuhjtman}

sergioyuhjtman@gmail.com

Universidad de Buenos Aires, FCEN, Departamento de Matem\'atica

Intendente Guiraldes 2160, Ciudad Universitaria, Pabell\'on I, C1428EGA

Ciudad Aut\'onoma de Buenos Aires, Argentina

\end{document}